\documentclass[a4paper,10pt,english]{amsart}

\usepackage{babel}

\usepackage{amsmath,amsfonts,amssymb,amsthm,mathrsfs,stmaryrd}
\usepackage{latexsym}

\newcommand{\N}{\mathbb{N}}

\DeclareMathAlphabet{\mathpzc}{OT1}{pzc}{m}{it}

\newcommand{\PI}[2]{\ensuremath{\boldsymbol\Pi^{#1}_{#2}}}
\newcommand{\SI}[2]{\ensuremath{\boldsymbol\Sigma^{#1}_{#2}}}
\newcommand{\DE}[2]{\ensuremath{\boldsymbol\Delta^{#1}_{#2}}}

\newcommand{\RCA}{{\ensuremath{\mathsf{RCA}_0}}}
\newcommand{\WKL}{{\ensuremath{\mathsf{WKL}_0}}}
\newcommand{\ACA}{{\ensuremath{\mathsf{ACA}_0}}}

\newcommand{\RT}{\ensuremath{\mathsf{RT}^1_{<\infty}}}
\newcommand{\RTt}{\ensuremath{\mathsf{RT}^2_2}}
\newcommand{\FUF}{\ensuremath{\mathsf{FUF}}}
\newcommand{\BS}{\ensuremath{\mathsf{B}\SI02}}

\newcommand{\PART}{\ensuremath{\mathsf{PART}}}

\theoremstyle{plain}
\newtheorem{theorem}{Theorem}[section]

\newtheorem{lemma}[theorem]{Lemma}

\theoremstyle{definition}
\newtheorem{definition}[theorem]{Definition}

\theoremstyle{remark}

\title[Linear extensions of partial orders and Reverse Mathematics]{Linear extensions of partial orders\\ and Reverse Mathematics}

\author{Emanuele Frittaion}
    \address{Dipartimento di Matematica e Informatica,
    Universit\`{a} di Udine,
    33100 Udine,
    Italy}
\email{emanuele.frittaion@uniud.it}

\author{Alberto Marcone}
    \address{Dipartimento di Matematica e Informatica,
    Universit\`{a} di Udine,
    33100 Udine,
    Italy}
\email{alberto.marcone@uniud.it}

\date{Saved: April 11, 2012.}

\begin{document}

\subjclass[2010]{Primary: 03B30; Secondary: 06A07}

\maketitle

\begin{abstract}
We introduce the notion of $\tau$-like partial order, where $\tau$ is
one of the linear order types $\omega$, $\omega^*$, $\omega+\omega^*$,
and $\zeta$. For example, being $\omega$-like means that every element
has finitely many predecessors, while being $\zeta$-like means that
every interval is finite. We consider statements of the form ``any
$\tau$-like partial order has a $\tau$-like linear extension'' and
``any $\tau$-like partial order is embeddable into $\tau$'' (when
$\tau$ is $\zeta$ this result appears to be new). Working in the
framework of reverse mathematics, we show that these statements are
equivalent either to \BS\ or to \ACA\ over the usual base system \RCA.
\end{abstract}


\section{Introduction}

Szpilrajn's Theorem (\cite{Szp30}) states that any partial order has a
linear extension. This theorem raises many natural questions, where in
general we search for properties of the partial order which are
preserved by some or all its linear extensions. For example it is
well-known that a partial order is a well partial order if and only if
all its linear extensions are well-orders.

A question which has been widely considered is the following: given a
linear order type $\tau$, is it the case that any partial order that
does not embed $\tau$ can be extended to a linear order that also does
not embed $\tau$? If the answer is affirmative, $\tau$ is said to be
extendible, while $\tau$ is weakly extendible if the same holds for any
countable partial order. For instance, the order types of the natural
numbers, of the integers, and of the rationals are extendible. Bonnet
(\cite{Bon69}) and Jullien (\cite{Jul}) characterized all countable
extendible and weakly extendible linear order types respectively.

We are interested in a similar question: given a linear order type
$\tau$ and a property characterizing $\tau$ and its suborders, is it
true that any partial order which satisfies that property has a linear
extension which also satisfies the same property? In our terminology:
does any $\tau$-like partial order have a $\tau$-like linear extension?
Here we address this question for the linear order types $\omega$,
$\omega^*$ (the inverse of $\omega$), $\omega+\omega^*$ and $\zeta$
(the order of integers). So, from now on, $\tau$ will denote one of
these.

\begin{definition}
Let $(P,\leq_P)$ be a countable partial order. We say that $P$ is
\begin{itemize}
  \item \emph{$\omega$-like} if every element of $P$ has finitely
      many predecessors;
  \item \emph{$\omega^*$-like} if every element of $P$ has finitely
      many successors;
  \item \emph{$\omega+\omega^*$-like} if every element of $P$ has
      finitely many predecessors or finitely many successors;
  \item \emph{$\zeta$-like} if for every pair of elements $x,y\in
      P$ there exist only finitely many elements $z$ with $x<_P
      z<_P y$.
\end{itemize}
\end{definition}

The previous definition resembles Definition 2.3 of Hirschfeldt and
Shore (\cite{HirSho07}), where linear orders of type $\omega$,
$\omega^*$ and $\omega+\omega^*$ are introduced. The main difference is
that the order properties defined by Hirschfeldt and Shore are meant to
uniquely determine a linear order type up to isomorphism, whereas our
definitions apply to partial orders in general and do not determine an
order type. Notice also that, for instance, an $\omega$-like partial
order is also $\omega+\omega^*$-like and $\zeta$-like.

We introduce the following terminology:

\begin{definition}
We say that $\tau$ is \emph{linearizable} if every $\tau$-like partial
order  has a linear extension which is also $\tau$-like.
\end{definition}

With this definition in hand, we are ready to formulate the results we
want to study:

\begin{theorem}\label{main}
The following hold:
\begin{enumerate}
 \item $\omega$ is linearizable;
 \item $\omega^*$ is linearizable;
 \item $\omega+\omega^*$ is linearizable;
 \item $\zeta$ is linearizable.
\end{enumerate}
\end{theorem}

A proof of the linearizability of $\omega$ can be found in Fra\"iss\'{e}'s
monograph (\cite[\S 2.15]{Fra00}), where the result is attributed to
Milner and Pouzet. $(2)$ is similar to $(1)$ and the proof of $(3)$
easily follows from $(1)$ and $(2)$. The linearizability of $\zeta$ is
apparently a new result (for a proof see Lemma \ref{lemma 2}
below).\medskip

In this paper we study the statements contained in Theorem \ref{main}
from the standpoint of reverse mathematics (the standard reference is
\cite{Sim09}), whose goal is to characterize the axiomatic assumptions
needed to prove mathematical theorems. We assume the reader is familiar
with systems such as \RCA\ and \ACA. The reverse mathematics of weak
extendibility is studied in \cite{DHLS} and \cite{Mon06}. The existence
of maximal linear extensions of well partial orders is studied from the
reverse mathematics viewpoint in \cite{MarSho11}.\medskip

Our main result is that the linearizability of $\tau$ is equivalent
over \RCA\ to the  $\SI02$ bounding principle \BS\ when $\tau \in
\{\omega, \omega^*, \zeta\}$, and to \ACA\ when $\tau=\omega+\omega^*$.
For more details on \BS, including an apparently new equivalent (simply
asserting that a finite union of finite sets is finite), see
\S\ref{Section FUF} below.

The linearizability of $\omega$ appears to be the first example of a
genuine mathematical theorem (actually appearing in the literature for
its own interest, and not for its metamathematical properties) that
turns out to be equivalent to \BS.\medskip

To round out our reverse mathematics analysis, we also consider a
notion closely related to linearizability:

\begin{definition}
We say that $\tau$ is \emph{embeddable} if every $\tau$-like partial
order $P$ embeds into $\tau$, that is there exists an order preserving
map from $P$ to $\tau$.\footnote{To formalize this definition in \RCA,
we need to fix a canonical representative of the order type $\tau$,
which we do in Definition 1.5.}
\end{definition}

It is rather obvious that $\tau$ is linearizable if and only if $\tau$
is embeddable. Let us notice that \RCA\ easily proves that embeddable
implies linearizable. Not surprisingly, the converse is not true. In
fact, we show that embeddability is strictly stronger when $\tau \in
\{\omega, \omega^*, \zeta\}$, and indeed equivalent to \ACA. The only
exception is given by $\omega+\omega^*$, for which both properties are
equivalent to \ACA.\medskip

We use the following definitions in \RCA.

\begin{definition}[\RCA]
Let $\leq$ denote the usual ordering of natural numbers. The linear
order $\omega$ is $(\N,{\leq})$, while $\omega^*$ is $(\N,{\geq})$.

Let $\{P_i\colon i\in Q\}$  be a family of partial orders indexed by a
partial order $Q$. The \emph{lexicographic sum} of the $P_i$ along $Q$,
denoted by  $\sum_{i\in Q}P_i$, is the partial order on the set
$\{(i,x)\colon i\in Q \land x\in P_i\}$  defined by
\[
(i,x) \leq (j,y) \iff i <_Q j \lor (i=j \land x \leq_{P_i} y).
\]

The \emph{sum} $\sum_{i<n}P_i$ can be regarded as the lexicographic sum
along the $n$-element chain. In particular $P_0+P_1$ is the
lexicographic sum along the $2$-element chain (and we have thus defined
$\omega + \omega^*$ and $\zeta = \omega^*+\omega$).

Similarly, the \emph{disjoint sum} $\bigoplus_{i<n} P_i$ is the
lexicographic sum along the $n$-element antichain.
\end{definition}


\section{$\SI02$ bounding and finite union of finite sets}\label{Section FUF}

Let us recall that \BS\ (standing for \SI02 bounding, and also known as
\SI02 collection) is the scheme:
\[
\tag{\BS} (\forall i<n) (\exists m) \varphi(i,n,m) \implies
(\exists k) (\forall i<n) (\exists m<k) \varphi(i,n,m),
\]
where $\varphi$ is any $\SI02$ formula.

It is well-known that \RCA\ does not prove \BS, which is strictly
weaker than \SI02 induction. Neither of \WKL\ and \BS\ implies the
other and Hirst (\cite{Hirst87}, for a widely available proof see
\cite[Theorem 2.11]{ChoJocSla01}) showed that \RTt\ (Ramsey theorem for
pairs and two colors) implies \BS.

A few combinatorial principles are known to be equivalent to \BS\ over
\RCA.

Hirst (\cite{Hirst87}, for a widely available proof see \cite[Theorem
2.10]{ChoJocSla01}) showed that, over \RCA, \BS\ is equivalent to the
infinite pigeonhole principle, i.e. the statement
\[
\tag{\RT} (\forall n) (\forall f:\N \to n) (\exists A \subseteq \N \text{ infinite}) (\exists c<n) (\forall m\in A) (f(m)=c).
\]
(The notation arises from viewing the infinite pigeonhole principle as
Ramsey theorem for singletons and an arbitrary finite number of
colors.)

Chong, Lempp and Yang (\cite{ChoLemYan10}) showed that a combinatorial
principle \PART\ about infinite $\omega+\omega^*$ linear orders,
introduced by Hirschfeldt and Shore (\cite[\S4]{HirSho07}), is also
equivalent to \BS. More recently, Hirst (\cite{Hirst12}) also proved
that \BS\ is equivalent to a statement apparently similar to Hindman's
theorem, but much weaker from the reverse mathematics viewpoint.

We consider the statement that a finite union of finite sets is finite:
\[
\tag{\FUF} (\forall i<n) (X_i \text{ is finite}) \implies
\bigcup_{i<n} X_i \text{ is finite}.
\]
Here ``$X$ is finite'' means $(\exists m)(\forall x \in X) (x<m)$. This
statement can be viewed as a second-order version of $\Pi_0$
regularity, which in the context of first-order arithmetic is known to
be equivalent to $\Sigma_2$ bounding (see e.g.\ \cite[Theorem
2.23.4]{HajPud}).

\begin{lemma}\label{lemma 0}
Over \RCA, \BS\ is equivalent to \FUF.
\end{lemma}
%
%
\begin{proof}
First notice that \FUF\ follows immediately from the instance of \BS\
relative to the \PI01, and hence \SI02, formula $(\forall x \in X_i)
(x<m)$.

For the other direction we use Hirst's result recalled above: it
suffices to prove that \FUF\ implies \RT. Let $f\colon\N\to n$ be
given. Define for each $i<n$ the set $X_i=\{m\colon f(m)=i\}$. Clearly
$\bigcup_{i<n}X_i=\N$ is infinite. By \FUF, there exists $i<n$ such
that $X_i$ is infinite. Now $X_i$ is an infinite homogeneous set for
$f$.
\end{proof}


\section{Linearizable types}

Notice that Szpilrajn's Theorem is easily seen to be computably true
(see \cite[Observation 6.1]{Dow98}) and provable in \RCA. We use this
fact several times without further notice.

We start by proving that \BS\ suffices to establish the linearizability
of $\omega$, $\omega^*$ and $\zeta$.

\begin{lemma}\label{lemma 1}
\RCA\ proves that \BS\ implies the linearizability of $\omega$ and
$\omega^*$.
\end{lemma}
\begin{proof}
We argue in \RCA\ and, by Lemma \ref{lemma 0}, we may assume \FUF. Let
us consider first $\omega$. So let $P$ be an $\omega$-like partial
order which, to avoid trivialities, we may assume to be infinite. We
recursively define a sequence $z_n\in P$ by letting $z_n$ be the least
(w.r.t.\ the usual ordering of $\N$) $x\in P$ such that $(\forall
i<n)(x\nleq_P z_i)$.

We show by \SI01 induction that $z_n$ is defined for all $n\in\N$.
Suppose that $z_i$ is defined for all $i<n$. We want to prove $(\exists
x\in P)(\forall i<n)(x\nleq_P z_i)$. Define $X_i=\{x\in P\colon x\leq
_P z_i\}$ for $i<n$. Since $P$ is $\omega$-like, each $X_i$ is finite.
By \FUF, $\bigcup_{i<n} X_i$ is also finite. The claim follows from the
fact that $P$ is infinite.

Now define for each $n\in\N$ the finite set
\[ P_n=\{x\in P\colon x\leq_P z_n\land(\forall i<n)(x\nleq_P z_i)\}.\]
It is not hard to see that the $P_n$'s form a partition of $P$, and
that if $x\leq_P y$ with $x\in P_i$ and $y\in P_j$, then $i\leq j$.
Then let $L$ be a linear extension of the lexicographic sum
$\sum_{n\in\omega} P_n$. $L$ is clearly a linear order and extends $P$
by the remark above. To prove that $L$ is $\omega$-like, note that the
set of $L$-predecessors of an element of $P_n$ is included in
$\bigcup_{i \leq n} P_i$, which is finite, by \FUF\ again.

For $\omega^*$, repeat the same construction using $\geq_P$ in place of
$\leq_P$, and let $L$ be a linear extension of
$\sum_{n\in\omega^*}P_n$.
\end{proof}

\begin{lemma}\label{lemma 2}
\RCA\ proves that \BS\ implies  the linearizability of $\zeta$.
\end{lemma}
\begin{proof}
In \RCA\ assume \FUF. Let $P$ be a $\zeta$-like partial order, which we
may again assume to be infinite. It is convenient to use the notation
$[x,y]_P = \{z\in P\colon x\leq_P z\leq_P y \lor y \leq_P z\leq_P x\}$,
so that $[x,y]_P \neq \emptyset$ if and only if $x$ and $y$ are
comparable.

We define by recursion a sequence $z_n\in P$ by letting $z_n$ be the
least (w.r.t.\ the ordering of $\N$) $x\in P$ such that
\[
x\notin \bigcup_{i,j<n}\mathopen{[}z_i,z_j\mathclose{]}_P.
\]

As before, since $P$ is infinite and $\zeta$-like, one can prove using
\SI01 induction and \FUF\ that $z_n$ is defined for every $n\in\N$. It
is also easy to prove that
\[
    P=\bigcup_{i,j\in\N}[z_i,z_j\mathclose{]}_P.
\]
Define for each $n\in\N$ the set
\[
P_n=\bigcup_{i<n}[z_i, z_n\mathclose{]}_P\setminus \bigcup_{i,j<n} [z_i,z_j\mathclose{]}_P.
\]
By \FUF, the $P_n$'s are finite. Moreover, they clearly form a
partition of $P$. Note also that $z_n\in P_n$ and every element of
$P_n$ is comparable with $z_n$. Furthermore, every interval $[x,y]_P$
is included in some $[z_i,z_j]_P$. Notice that the same holds for any
partial order extending $\leq_P$.

We now extend $\leq_P$ to a partial order $\preceq_P$ such that any
linear extension of $(P,{\preceq_P})$ is $\zeta$-like. We say that $n$
is left if $z_n\leq_P z_i$ for some $i<n$; otherwise, we say that $n$
is right. Notice that, since $z_n \in P_n$, $n$ is right if and only if
$z_i\leq_P z_n$ for some $i<n$ or $z_n$ is incomparable with every
$z_i$ with $i<n$.

The order $\preceq_P$ places $P_n$ below or above every $P_i$ with
$i<n$ depending on whether $n$ is left or right. Formally, for $x,y\in
P$ such that $x\in P_n$ and $y\in P_m$ let
\[
x\preceq_P y \iff
(n=m\land x\leq_P y)\lor(n<m\land m\ \text{is right})\lor(m<n\land n\ \text{is left}).
\]

We claim that $\preceq_P$ extends $\leq_P$. Let $x\leq_P y$ with $x\in
P_n$ and $y\in P_m$. If $n=m$, $x\preceq_P y$ by definition. Suppose
now that $n<m$, so that we need to prove that $m$ is right. As $x\in
P_n$, $z_i\leq_P x$ for some $i\leq n$. Since $y\in P_m$,  $y$ is
comparable with $z_m$. Suppose that  $z_m<_P y$. Then $y\leq_P z_j$ for
some $j<m$, and so $z_i\leq_Px\leq_P y\leq_P z_j$ with $i,j<m$,
contrary to $y\in P_m$. It follows that $y\leq_P z_m$ and thereby
$z_i\leq_P z_m$ with $i<m$. Therefore, $m$ is right, as desired. The
case $n>m$ (where we need to prove that $n$ is left) is similar.

We claim that $(P,\preceq_P)$ is still $\zeta$-like. To see this, it is
enough to show that for all $i,j<n$
\[
\{x\in P\colon z_i\preceq_P x\preceq_P z_j\} \subseteq \bigcup_{k<n} P_k
\]
and apply \FUF. Let $x\in P_k$ be such that $z_i\prec_P x\prec_P z_j$.
Suppose, for a contradiction, that $k\geq n$ and hence that $i,j<k$. By
the definition of $\preceq_P$, $z_i\prec_P x$ implies that $k$ is
right. At the same time,  $x\prec_P z_j$ implies that $k$ is left, a
contradiction.

Now let $L$ be any linear extension of $(P,{\preceq_P})$ and hence of
$(P,{\leq_P})$. We claim that $L$ is $\zeta$-like. To prove this, we
show that for all $i,j\in\N$
\[
\{ x\in P\colon z_i\leq_L x\leq_L z_j\}=\{x\in P\colon z_i\preceq_P x\preceq_P z_j\}.
\]
One inclusion is obvious because $\leq_L$ extends $\preceq_P$. For the
converse, observe that the $z_n$'s are $\preceq_P$-comparable with any
other element.
\end{proof}

We can now state and  prove our reverse mathematics results.

\begin{theorem}\label{theorem 0}
Over \RCA, the following are pairwise equivalent:
\begin{enumerate}
 \item \BS;
 \item $\omega$ is linearizable;
 \item $\omega^*$ is linearizable;
 \item $\zeta$ is linearizable.
\end{enumerate}
\end{theorem}
\begin{proof}
Lemma \ref{lemma 1} gives $(1)\rightarrow(2)$ and $(1)\rightarrow(3)$.
The implication $(1)\rightarrow(4)$ is Lemma \ref{lemma 2}.

To show $(2)\rightarrow(1)$, we assume linearizability of $\omega$ and
prove \FUF. So let $\{X_i\colon i<n\}$ be a finite family of finite
sets. We define $P= \bigoplus_{i<n} (X_i+\{m_i\})$, where the $m_i$'s
are distinct and every $X_i$ is regarded as an antichain. $P$ is
$\omega$-like, and so by $(2)$ there exists an $\omega$-like linear
extension $L$ of $P$. Let $m_j$ be the $L$-maximum of $\{m_i \colon
i<n\}$. Then $\bigcup_{i<n} X_i$ is included in the set of
$L$-predecessors of $m_j$, and is therefore finite because $L$ is
$\omega$-like.

The implication $(3)\to(1)$ is analogous. For $(4)\to(1)$, prove \FUF\
by using the partial order $\bigoplus_{i<n} (\{\ell_i\} + X_i +
\{m_i\})$.
\end{proof}

We now show that the linearizability of $\omega+\omega^*$ requires
\ACA.

\begin{theorem}\label{theorem 1}
Over \RCA, the following are equivalent:
\begin{enumerate}
 \item \ACA;
 \item $\omega+\omega^*$ is linearizable.
\end{enumerate}
\end{theorem}
\begin{proof}
We begin by proving $(1)\to(2)$. Let $P$ be an $\omega+\omega^*$-like
partial order. In \ACA\ we can define the set $P_0$ of the elements
having finitely many predecessors. So $P_1=P\setminus P_0$ consists of
elements having finitely many successors. Clearly, $P_0$ is
$\omega$-like and $P_1$ is $\omega^*$-like. Since \ACA\ is strong
enough to prove \BS, by Lemma \ref{lemma 1}, $P_0$ has an $\omega$-like
linear extension $L_0$ and $P_1$ has an $\omega^*$-like linear
extension $L_1$. Since $P_0$ is downward closed and $P_1$ is upward
closed, it is not difficult to check that the linear order $L=L_0+L_1$
is $\omega+\omega^*$-like and extends $P$.\smallskip

For the converse, let $f\colon\N\to\N$ be a one-to-one function. We set
out to define an $\omega+\omega^*$-like partial order $P$ such that any
$\omega+\omega^*$-like linear extension of $P$ encodes the range of
$f$. To this end, we use an $\omega+\omega^*$-like linear order
$A=\{a_n\colon n\in\N\}$ given by the false and true stages of $f$.
Recall that $n\in\N$ is said to be true (for $f$) if $(\forall
m>n)(f(m)>f(n))$ and false otherwise, and note that the range of $f$ is
\DE01 definable from any infinite set of true stages.

The idea for $A$ comes from the well-known construction of a computable
linear order such that any infinite descending sequence computes
$\emptyset'$. This construction can be carried out in \RCA\ (see
\cite[Lemma 4.2]{MarSho11}). Here, we define $A$ by letting $a_n\leq
a_m$ if and only if either
\begin{gather*}
f(k)<f(n) \text{ for some }n<k\leq m, \text{ or}\\
 m\leq n \text{ and } f(k)>f(m) \text{ for all }m<k\leq n.
\end{gather*}

It is not hard to see that $A$ is a linear order. Moreover, if $n$ is
false, then $a_n$ has finitely many predecessors and infinitely many
successors. Similarly, if $n$ is true, then $a_n$ has finitely many
successors and infinitely many predecessors. In particular, $A$ is an
$\omega+\omega^*$-like linear order.

Now let $P=A\oplus B$ where $B=\{b_n\colon n\in\N\}$ is a linear order
of order type $\omega^*$, defined by letting $b_n \leq b_m$ if and only
if $n \geq m$. It is clear that $P$ is an $\omega+\omega^*$-like
partial order. By hypothesis, there exists an $\omega+\omega^*$-like
linear extension $L$ of $P$. We claim that $n$ is a false stage if and
only if it satisfies the \PI01 formula $(\forall m)(a_n<_L b_m)$.

In fact, if $n$ is false and $b_m\leq_L a_n$, then $b_m$ has infinitely
many successors in $L$, since $a_n$ has infinitely many successors in
$P$ and a fortiori in $L$. On the other hand, $b_m$ has infinitely many
predecessors in $P$, and hence also in $L$, contradiction. Likewise, if
$n$ is true and $a_n<_L b_m$ for all $m$, then $a_n$ has infinitely
many successors as well as infinitely many predecessors in $L$, which
is a contradiction again.

Therefore, the set of false stages is \DE01, and so is the set of true
stages, which thus exists in \RCA. This completes the proof.
\end{proof}


\section{Embeddable types}

We turn our attention to embeddability. As noted before, \RCA\ suffices
to prove that ``$\tau$ is embeddable'' implies ``$\tau$ is
linearizable''. The converse is true in \ACA.  Actually, embeddability
is  equivalent to \ACA. We thus prove the following.

\begin{theorem}
The following are pairwise equivalent over \RCA:
\begin{enumerate}
 \item \ACA;
 \item $\omega$ is embeddable;
 \item $\omega^*$ is embeddable;
 \item $\zeta$ is embeddable;
 \item $\omega+\omega^*$ is embeddable;
\end{enumerate}
\end{theorem}
\begin{proof}
We first show that $(1)$ implies the other statements. Since \BS\ is
provable in \ACA, it follows from Theorem \ref{theorem 0} that \ACA\
proves the linearizability of $\omega$, $\omega^*$ and $\zeta$.  By
Theorem \ref{theorem 1}, \ACA\ proves the linearizability of $\omega
+\omega^*$. We now claim that in \ACA\ ``$\tau$ is linearizable''
implies ``$\tau$ is embeddable'' for each $\tau$ we are considering.
The key fact is that the property of having finitely many predecessors
(successors) in a partial order, as well as having exactly $n\in\N$
predecessors (successors), is arithmetical. Analogously, for a set, and
hence for an interval, being finite or having size exactly $n\in\N$ is
arithmetical too. (All these properties are in fact $\SI02$.)

We consider explicitly the case of $\omega+\omega^*$ (the other cases
are similar). So let $L$ be a $\omega+\omega^*$-like linear extension
of a given $\omega+\omega^*$-like partial order. We want to show that
$L$ is embeddable into $\omega+\omega^*$. Define $f\colon L\to
\omega+\omega^*$ by
\[ f(x)=\begin{cases}
          (0,|\{y\in L\colon y<_L x\}|) & \text{if}\ x\ \text{has finitely many predecessors,}\\
          (1,|\{y\in L\colon x<_L y\}|) &  \text{otherwise}.
        \end{cases}\]
It is easy to see that $f$ preserves the order.\smallskip

For the reversals, notice that $(5)\to(1)$ immediately follows from
Theorem \ref{theorem 1}.

As the others are quite similar, we only prove $(2)\to(1)$ with a
construction similar to that used in the proof of Theorem 3.1 in
\cite{FriHir90}. Let $f\colon\N\to\N$ be a given one-to-one function.
We want to prove that the range of $f$ exists. We fix an antichain
$A=\{a_m\colon m\in\N\}$ and elements $b^n_j$ for $n\in\N$ and $j\leq
n$. The partial order $P$ is obtained by putting for each $n\in\N$ the
$n+1$ elements $b^n_j$ below $a_{f(n)}$. Formally, $b^n_j \leq_P a_m$
when $f(n) \leq m$, and there are no other comparabilities.

$P$ is clearly an $\omega$-like partial order. Apply the hypothesis and
obtain an embedding $h\colon P\to\omega$. Now, we claim that $m$
belongs to the range of $f$ if and only if $(\exists
n<h(a_m))(f(n)=m)$. One implication is trivial. For the other, suppose
that $f(n)=m$. By construction, $a_m$ has at least $n+1$ predecessors
in $P$, and thus it must be $h(a_m)>n$.
\end{proof}

\end{document}